\documentclass[12pt]{article}
\usepackage{amsmath, amsfonts, amsthm}
\usepackage{fancyhdr}
\usepackage{fullpage}
\usepackage{physics}
\usepackage{mathtools}
\usepackage{hyperref}
\usepackage{amssymb}
\usepackage{mathrsfs}
\usepackage{array}
\usepackage{amscd}
\usepackage{mathrsfs}
\usepackage[all,cmtip]{xy}
\usepackage{lmodern}
\usepackage[Symbol]{upgreek}
\usepackage{bm}
\usepackage[nopar]{lipsum}
\usepackage{rotating}
\usepackage{tikz}
\usepackage{calc}
\usepackage{tikz-cd}
\usetikzlibrary{decorations.pathmorphing}

\usepackage[utf8]{inputenc}
\usepackage[english]{babel}
\usepackage{t1enc}
\usepackage[mathscr]{eucal}
\usepackage{indentfirst}
\usepackage{graphicx}
\usepackage{graphics}
\usepackage{pict2e}
\usepackage{epic}
\usepackage[margin=2.9cm]{geometry}
\usepackage{epstopdf}

\hypersetup{
    colorlinks=true,
    linkcolor=blue,
    filecolor=magenta,      
    urlcolor=blue,
    citecolor=blue
}
 
\newtheoremstyle{exampstyle}
  {5pt} % Space above
  {5pt} % Space below
  {} % Body font
  {} % Indent amount
  {} % Theorem head font
  {.} % Punctuation after theorem head
  {.5em} % Space after theorem head
  {} % Theorem head spec (can be left empty, meaning `normal')

\newtheorem{theorem}{Theorem}[section]

\newtheorem{definition}{Definition}[section]
\newtheorem*{axiom*}{Axiom}

\newtheorem{proposition}{Proposition}[section]
\newtheorem{lemma}{Lemma}[section]
\newtheorem{corollary}{Corollary}[section]

\newtheoremstyle{exampstyle2}
  {5pt} % Space above
  {5pt} % Space below
  {\small \slshape} % Body font
  {} % Indent amount
  {\normalsize\itshape} % Theorem head font
  {.} % Punctuation after theorem head
  {.5em} % Space after theorem head
  {} % Theorem head spec (can be left empty, meaning `normal')

\theoremstyle{exampstyle2}
\newtheorem*{example}{Example}

\newcommand{\ds}{\displaystyle}

\newcommand{\is}{\hspace{2pt}}
\newcommand{\dx}{\is dx}

\newcommand{\R}{\mathbb{R}}
\newcommand{\C}{\mathbb{C}}

\newcommand{\Proj}{\mathbb{P}}

\begin{document}
%\maketitle

\title{\textbf{The Complex Geometry and Representation Theory of Statistical
Transformation Models}}\author{Shuhao Li}\date{}\maketitle

%\tableofcontents

%\newpage

\begin{abstract}

Given a measure space ${\mathcal X}$, we can construct a number of induced structures: eg. its $L^2$ space, the space ${\mathcal P}({\mathcal X})$ of probability distributions on ${\mathcal X}$. If, in addition, ${\mathcal X}$ admits a transitive measure-preserving Lie group action, natural actions are induced on those structures. We expect relationships between these induced structures and actions. We study, in particular, the relations between $L^2({\mathcal X})$ and exponential transformation models on ${\mathcal X}$, which are special “submanifolds” of ${\mathcal P}({\mathcal X})$ closed under the induced action, whose tangent bundles are Kähler manifolds (given by Molitor). Geometrically,  we show the tangent bundle has, locally, the “same” Kähler metric with the Fubini-Study metric on the projectivization of $L^2({\mathcal X})$. Moreover we show the action on the tangent bundle is equivariant with that on $L^2({\mathcal X})$, which is a unitary representation. Finally, in some cases, when the symplectic action on the tangent bundle is Hamiltonian, we show that any coadjoint orbit in the image of its moment map induces, via Kirillov’s correspondence from orbit method, irreducible unitary representations that are sub-representations of the aforementioned representation in $L^2({\mathcal X})$.

% Given a measure space ${\mathcal X}$, we can construct a number of induced structures: eg. its $L^2$ space, the space ${\mathcal P}({\mathcal X})$ of probability distributions on ${\mathcal X}$. If, in addition, ${\mathcal X}$ admits a transitive measure-preserving Lie group action, natural actions are induced on those structures. We expect relationships between these induced structures and actions. We study, in particular, the relations between $L^2({\mathcal X})$ and exponential transformation models on ${\mathcal X}$, which are special “submanifolds” of ${\mathcal P}({\mathcal X})$ closed under the induced action, whose tangent bundles are Kähler manifolds (given by Molitor). Geometrically,  we show the tangent bundle has, locally, the “same” Kähler metric with the Fubini-Study metric on the projectivization of $L^2({\mathcal X})$. Moreover we show the action on the tangent bundle is equivariant with that on $L^2({\mathcal X})$, which is a unitary representation. Finally, in some cases, when the symplectic action on the tangent bundle is Hamiltonian, we show that any coadjoint orbit in the image of its moment map induces, via Kirillov’s correspondence from orbit method, irreducible unitary representations that are sub-representations of the aforementioned representation in $L^2({\mathcal X})$. 
\end{abstract}

\section{Introduction}

% Given a Lie group $G$ and a transitive $G$-invariant measure space $\mathcal{X}$ with invariant measure $\mu$,
% From the measure space $\mathcal{X}$ 
Given a measure space $\mathcal{X}$ with measure $\mu$,
we can construct a number of natural algebraic or geometric structures from the measure space structure of $\mathcal{X}$: for example, the $L^2$ space of $(\mathcal{X},\mu)$ with the Hilbert space structure, the space $\mathcal{P}(\mathcal{X})$ of all probability distributions over $\mu$ with the Fisher information metric, and in case $\mathcal{X}$ is a Riemannian manifold with the geometric measure from the metric, the cotangent bundle $T^*\mathcal{X}$ with the symplectic structure. In particular, the field of information geometry studies statistical manifolds, which are, in some sense, finite-dimensional immersed submanifolds of the space $\mathcal{P}(\mathcal{X})$ (which is in general infinite dimensional). These statistical manifolds have the structure of a Riemannian manifold (with the Fisher metric) and in addition a compatible pair of affine connections, called the exponential and mixture connections. A class of particularly nice examples of statistical manifolds are the exponential families. In \cite{molitor2013exponential}, Molitor proves that the tangent bundle of a statistical manifold can be naturally given a K\"ahler structure using a construction given by Dombrowski in \cite{dombrowski1962geometry}, where Molitor demonstrates the connection of this construction with quantum mechanics.

If, in addition, a Lie group $G$ acts transitively on $\mathcal{X}$ and the measure $\mu$ is a $G$-invariant measure, this structure-preserving $G$-action can naturally give rise to actions on those algebraic and geometric structures built upon the measure space structure of $\mathcal{X}$: $G$ acts on the $L^2$ space by left translation, i.e. a function $f:\mathcal{X}\to \C$ is mapped to a function $gf$ by $g\in G$ such that $(gf)(x)=f(g^{-1}x)$ for any $x\in \mathcal{X}$. This action is in fact a unitary representation of $G$. A similar $G$-action on $\mathcal{P}(\mathcal{X})$ can be defined. Finally in case $\mathcal{X}$ is a Riemannian manifold with geometric measure such that $G$ acts on $\mathcal{X}$ by isometry in the Riemannian metric, $G$ can act on the cotangent bundle $T^*\mathcal{X}$ by pullback maps. If these new structures induced from $\mathcal{X}$ only depends on the measure space structure of $\mathcal{X}$ and the action of $G$ on $\mathcal{X}$ leaves the measure unchanged, then it is intuitive that these induced actions of $G$ on the new structures are structure-preserving, which we will make precise for some of the cases in the following sections. In particular, in e.g. \cite{barndorff2012parametric,barndorff2012decomposition}, Barndorff-Nielsen studies statistical transformation models, which are nice statistical manifolds in $\mathcal{P}(\mathcal{X})$ that are closed under the actions of $G$. Given the work by Molitor, we are particularly interested in exponential transformation models, that is, statistical transformation models which are also exponential families. The action of $G$ on the exponential transformation model gives rise to a $G$-action on the tangent bundle of the exponential transformation model by pushforward maps, and this induced action can be shown to be an action by isometries in the K\"ahler metric. Thus it is, in particular, a symplectic action.

It is then reasonable to expect some relationships between these induced structures from the measure space $\mathcal{X}$ and the induced structure-preserving $G$-actions on the different structures. In this work, we are interested in exploring the relationship between the tangent bundle $TM$ of an exponential transformation model $M$ with the $G$-action by K\"ahler isometry, and the $L^2$ space of $\mathcal{X}$ with the unitary $G$-representation, under some special circumstances. We shall show, generalizing from Molitor's work in \cite{molitor2013exponential}, that there is a $G$-equivariant map from $TM$ to $L^2(\mathcal{X})$, such that if we compose the map with the canonical projection map $L^2(\mathcal{X})\to \Proj L^2(\mathcal{X})$, the resulting map respects the K\"ahler structures of $TM$ and the Fubini-Study K\"ahler metric of $\Proj L^2(\mathcal{X})$. 
Furthermore  we assume that $\mathcal{X}$ has a normal stabilizer subgroup at some point in $\mathcal{X}$, and that the group is a semisimple nilpotent (or more generally exponential solveable) Lie group. We further assume that the $G$-action on the tangent bundle is not only symplectic but also Hamiltonian. Let $\mathfrak{g}^*$ be the dual Lie algebra of $G$. Kirillov's orbit method (see e.g. \cite{kirillov2004lectures,folland2016course}) establishes a one-to-one correspondence between irreducible unitary representations of $G$ and coadjoint orbits of $G$ (i.e. the orbits in the coadjoint action of $G$ on $\mathfrak{g}^*$). 
Now in our situation, the moment map $\mu:TM\to \mathfrak{g}^*$ of the Hamiltonian action takes $TM$ to certain coadjoint orbits, which induces unitary irreducible representations of $G$. We will show that, if we add some more constraints to the Hamiltonian action, these representations can be realized as subrepresentations of the aforementioned unitary representation of $G$ in $L^2(\mathcal{X})$.
In general, following this principle, we can expect more relationships between the geometric and algebraic structures induced from $\mathcal{X}$, which could reveal interesting representation theoretic properties of $G$.

\bigskip

\textbf{Acknowledgement.} I would like to thank my advisor Professor Renato Feres for introducing me to this topic, and for the countless hours guiding me to develop this work. I would like to thank Professor Klaus Mohnke for pointing out the relation between the symplectic form in the Dombrowski's construction and the canonical symplectic form on cotangent bundles. A large portion of this work is carried out during the Freiwald Scholars Program in the summers of 2019 and 2020.
% //where $G$ acts on $\mathcal{X}$ by isometry.
% This structure-preserving action of $G$ on $\mathcal{X}$ induces a number of $G$-actions on other algebraic or geometric structures related to the measure space $\mathcal{X}$: for example, $G$ acts on the cotangent bundle $T\mathcal{X}$

\section{Background: Exponential Families}

This section presents definitions needed from information geometry and Molitor's construction of Kählerification, following \cite{amari2007methods} and \cite{molitor2013exponential}.

\subsection{Information Geometry}

% Information geometry considers the geometric structure on families of probability distributions.
Let $\mathcal{X}$ be a measure space with measure $\mu$. We let $\mathcal{P}(\mathcal{X})$ be the space of all probability distributions on $\mathcal{X}$ under $\mu$:
\begin{equation}
  \mathcal{P}(\mathcal{X}) = \{p:\mathcal{X}\to \R \mid p\text{ is measurable}, p\geq 0, \int_{\mathcal{X}} p(x) \,d\mu = 1\}.
\end{equation}

\begin{example}
If $\mathcal{X}$ is a finite set with the counting measure, we index elements of $\mathcal{X}$ from $1$ to $n$, and denote the probability of the $i$-th element to be $p_i$. Then we have
\begin{equation}
  \mathcal{P}_n:=\mathcal{P}(\mathcal{X}) = \{p:\mathcal{X}\to \R\mid 0\leq p_i \leq 1, \sum_{i=1}^n p_i = 1\}.
\end{equation}
This set can be viewed as a $(n-1)$-dimensional manifold with boundary.

\end{example}
In general, of course, $\mathcal{P}(\Omega)$ may be infinite-dimensional.
A statistical model, loosely speaking, is a parametrized (finite-dimensional) submanifold of $P(\mathcal{X})$. 
\begin{definition}
Let $\mathcal{X}$ be a measure space with measure $\mu$. A \textbf{statistical model} is a $n$-dimensional differentiable manifold $M$ together with a map $j:M\to \mathcal{P}(\mathcal{X})$ such that 
\begin{enumerate}
  \item $j$ is one-to-one;
  % \item (Regularity) for all $\xi \in S$, if $\xi^1,\dots, \xi^n$ are coordiante functions for a chart in a neighborhood of $\xi$, the functions $(\frac{\partial j(\xi)}{\partial \xi^1},\dots, \frac{\partial j(\xi)}{\partial \xi^n})$ are linearly independent (if we think of $P(\Omega)$ as a manifold, this says that 
  \item  For $\xi\in M$, $1\leq i\leq n$, define functions $(\partial_i  j)_{\xi}$ on $\mathcal{X}$ as follows: pick a smooth path $\gamma(t)$ realizing $(\frac{\partial}{\partial \xi^i})_{\xi}\in T_\xi M$; then for each $x\in \mathcal{X}$,
  \begin{align*}
   (\partial_i j)_{\xi} (x)=\frac{d}{dt}\bigg|_{t=0}\big[j(\gamma(t))\big](x).
   \end{align*} We require $(\partial_1 j)_\xi,\dots, (\partial_n j)_\xi$ to be $\R$ linearly independent as functions on $\mathcal{X}$, for each $\xi\in M$ (i.e. $dj_\xi$ is one-to-one for every $\xi\in M$).
  % ).
\end{enumerate}
\end{definition}
For the most common examples of (finite-dimensional) statistical models, there is usually a global chart $M\to E$ (which we will always assume, for simplicity), where $E$ is an open set of $\R^n$. With some abuse of notations, we will usually write a statistical model as
\begin{equation}
  M=\{p(x;\xi)\mid \xi=(\xi^1,\dots, \xi^n)\in E\subset_{\text{open}} \R^n \},
\end{equation}
or simply write $\xi\in \R^n$ as elements of $S$.
\begin{example}[Finite sets]
Suppose $\Omega = \{x_1,\dots, x_n\}$ is given the counting measure.
Let $E\subset \R^{n-1}$ be the set 
\begin{equation}
E=\{(\xi^1,\dots, \xi^{n-1})\mid \xi^i>1 \text{ for all }i, \text{ and }\sum_{i=1}^{n-1}\xi^i<1\}.
\end{equation}
The statistical model $\mathcal{P}_n^\times$ is defined as
\begin{equation}
  \mathcal{P}_n^\times = \{p_\xi:\Omega\to \R\mid \xi\in E\},
\end{equation}
where 
\begin{equation}
  p_\xi(x_i) = 
  \begin{cases}
    \xi^i &i=1,\dots, n-1\\
    1-\sum_{i=1}^{n-1}\xi^i &i=n
  \end{cases}
  .
\end{equation}
$\mathcal{P}_n^\times$ is $\mathcal{P}_n$ without the boundary. The map $j$ can easily be checked to satisfy the two requirements since it can be viewed as an inclusion map $j: \mathcal{P}_n^\times \hookrightarrow \mathcal{P}_n$.
\end{example}

% In this paper, we will always impose the following restriction on the set of statistical models we consider.

\begin{example}[Exponential Families]
An exponential family $\mathcal{E}$ on a measure space $\mathcal{X}$ with measure $\mu$ is a statistical model
\begin{equation}
  M = \Bigg\{p(x;\xi) = \exp{C(x)+\sum_{i=1}^n \xi^i F_i(x) - \psi(\xi)} \mid \xi \in E\subset_{\text{open}} \R^n\Bigg\},
\end{equation}
where $C(x), F_1(x),\dots, F_n(x):\mathcal{X}\to \R$ are measurable functions, $\{1, F_1(x),\dots, F_n(x)\}$ are linearly independent, and $\psi$ is the normalizing constant
\begin{equation}
   \psi(\xi) = \ln(\int_\mathcal{X} \exp{C(x)+\sum_{i=1}^n \xi^i F_i(x)}\dx).
 \end{equation} 
\end{example}
Many common statistical models are exponential families. For example, $\mathcal{P}_n^\times$ is an exponential family. For any appropriate $\xi$, we have
\begin{equation}
  p_\xi(x_j)=\exp{\sum_{i=1}^n \xi^i F_i(x) - \psi(\xi)},
\end{equation}
where $F_i(x_j)=\delta_{ij}$, and $\psi(\xi)$ is some appropriate normalization function. Other important examples of exponential families are Gaussian, binomial, exponential, Poisson. In this paper we will always consider exponential families.

\bigskip
Next, we will introduce certain quantities that are meaningful in the context of probability, which give a geometric structure to the manifold. We introduce the Fisher-information matrix on a statistical model $M$, which will give the Riemannian geometric structure on the statistical model as a manifold.
% The Fisher information matrix at a point $\xi$ in $S$ is the second moment of the log-likelihood functions: 
% The Fisher-information matrix at a point $\xi$ in $S$ is the covariance matrix of the ``score function'', which measures how much the log-likelihood function changes with respect to a change in the parameter space.
% \begin{definition}
%   The \textbf{log-likelihood function} at $\xi\in S$ is a function $l_\xi:\Omega\to \R$ such that
%   \begin{equation}
%     l_\xi(x)=\ln p(x;\xi).
%   \end{equation}
%   The $j$-th \textbf{score function} at $\xi\in S$ is the function
%   \begin{equation}
%     \partial_j l_\xi(x) = \frac{\partial}{\partial \xi^j} \ln(p(x;\xi)). 
%   \end{equation}
% \end{definition}
% \begin{remark}\label{prop:fishinfo}
%   Under certain standard conditions (that we shall always take as assumptions), the expectation of a score function is $0$ at any $\xi\in S$:
%   \begin{equation}
%     \mathbb{E}[\partial_j l_\xi(x)]=\int_\Omega \partial_j l_\xi(x) \is p(x;\xi) \dx =0.
%   \end{equation}
% \end{remark}

  \begin{definition}
    The \textbf{Fisher-information matrix} at $\xi \in M$ is the matrix with entries $\mathbb{E}[\partial_i l_\xi(x) \partial_j l_\xi(x)]$ at the $(i,j)$ position.
  \end{definition}

% Note that with Remark \ref{prop:fishinfo}, the Fisher-information matrix is simply the covariant matrix of the score functions.

% [Give interpretation from MLE]

 \begin{theorem}
   The bilinear mapping $h_\xi:T_\xi M\times T_\xi M\to \R$ on $\xi \in M$ given by
   \begin{equation}
     h_\xi(\frac{\partial}{\partial\xi^i},\frac{\partial}{\partial\xi^j}) = h_{ij}=\mathbb{E}[\partial_i l_\xi(x)\partial_j l_\xi(x)]
   \end{equation}
   is a Riemannian metric (i.e. the Fisher-information matrix is symmetric and positive-definite). This metric on $M$ is referred to as the \textbf{Fisher metric}.
 \end{theorem}

A statistical model has more structure than a general Riemannian manifold. It has a two non-metrical connections $\nabla^{(1)}$ and $\nabla^{(-1)}$ that are related in a special way.
 \begin{definition}
   We give the Christoffel symbols of $\nabla^{(1)}$ and $\nabla^{(-1)}$:
   \begin{align}
     \Gamma^{(1)}_{ij,k} &= \mathbb{E}[(\partial_i \partial_j l_\xi)(\partial_k l_\xi)];\\
     \Gamma^{(-1)}_{ij, k}&= \mathbb{E}[(\partial_i \partial_j l_\xi)(\partial_k l_\xi)]+\mathbb{E}[(\partial_i l_\xi \partial_j l_\xi \partial_k l_\xi)].
   \end{align}
 \end{definition}

 % This structure generalizes to the following definition:
 The pair of connections satisfy an analogue of the Leibniz law, which we formally state as follows.
 \begin{definition}
   A \textbf{dualistic structure} on a Riemannian manifold $M$ with metric $g$ is a pair $(\nabla, \nabla^*)$ of connections satisfying the compatibility condition with respect to the metric: for all $X,Y,Z\in \Gamma(TM)$,
   \begin{equation}
     Z(g(X,Y)) = h(\nabla_Z X, Y) + h(X, \nabla^*_Z Y).
   \end{equation}
 \end{definition}
 % For example, the Levi-Civita connection $\nabla$ is dual to itself, itself satisfying the Leibniz law.

\begin{definition}
  A \textbf{statistical manifold} is a Riemannian manifold with a dualistic structure $(\nabla, \nabla^*)$, with both connections non-metrical.
\end{definition}

 \begin{theorem}
   On a statistical model $S$, the two connections $\nabla^{(1)}$ and $\nabla^{(-1)}$ form a dualistic structure with respect to the Fisher metric. That is, a statistical model is a special statistical manifold.
 \end{theorem}

 \begin{definition}
   For a Riemannian manifold with a dualistic structure $(M,g, \nabla, \nabla^*)$, we say $M$ is \textbf{dually flat} if both $\nabla$ and $\nabla^*$ are flat, in the sense that both curvature and torsion vanish.
 \end{definition}
 % Thus the definition of a statistical model is a 

 The main result that we are going to use later is this:
 \begin{theorem}
   An exponential family $(\mathcal{E}, h, \nabla^{(1)}, \nabla^{(-1)})$ is dually flat.
 \end{theorem}

 \subsection{Tangent Bundles of Exponential Families}

 Given an exponential family with its canonical dualistic structure $(\mathcal{E}, h, \nabla^{(1)}, \nabla^{(-1)})$, we consider the geometric structure on its tangent bundle $T\mathcal{E}$. The goal is to show that $T\mathcal{E}$ has a canonical K\"ahler structure. The construction follows \cite{dombrowski1962geometry}, \cite{molitor2013exponential}.

 \begin{theorem}[Dombrowski Splitting Theorem, \cite{dombrowski1962geometry}]
   For a manifold $M$ with an affine connection $\nabla$, 
   \begin{equation}
     T_{v_x}(TM)\simeq T_xM \oplus T_xM,\quad\forall x\in M, v_x\in T_xM.
   \end{equation}
   % where $\simeq$ is vector bundle isomorphism.
 \end{theorem}
 The isomorphism $\Phi: T_{v_x}(TM)\to  T_xM\oplus T_xM$ can be written as follows. Let $\pi:TM\to M$ and $\tilde \pi: T(TM)\to TM$. Let $U$ be a ``nice'' neighborhood of $x$; let $\tau:\pi^{-1}(U)\to T_x M$ be the map that takes a tangent vector $\theta'\in T_pM$, where $p\in U$, to an element of $T_x M$, such that $\tau(\theta')$ is the tangent vector obtained from parallel translation through the unique geodesic from $p$ to $x$. Now, let $K:T(TM)\to TM$ be defined as
 \begin{equation}
   K:Z\in T(TM) \mapsto \lim_{t\to 0}\frac{\tau(\tilde Z(t)) - \tilde\pi(Z)}{t},
 \end{equation}
where $\tilde Z(t)\subset TM$ is the curve in $TM$ that realizes $Z$ (i.e. $\tilde Z(0) = \tilde \pi(Z)$, $\frac{d}{dt}\Big|_{t=0} \tilde Z(t)= Z$).
For a point $Z_{\theta_x}$ where $\tilde \pi(Z_{\theta_x})=\theta_x\in TM$ and $\pi(\theta_x)=x\in M$, it is mapped by $\Phi$ to
\begin{equation}
  \Phi: Z_{\theta_x}\in T(TM)\mapsto (\theta_x, (\pi_*)_{\theta_x}(Z_{\theta_x}), K(Z_{\theta_x})).
\end{equation}
The second component is called the horizontal component of $Z_{\theta_x}$, and the third component is  called the vertical component of $Z_{\theta_x}$. With certain abuse of notation, we now denote elements in $T(TM)$ to be the triple $(\theta_x, l_x, v_x)$, where $\theta_x=\tilde\pi(\theta_x, l_x, v_x)$, and $l_x\in T_{\theta_x}M$ (resp. $v_x\in T_{\theta_x}M$) the horizontal (resp. vertical) component of $(\theta_x, l_x, v_x)$.

In particular, we have the direct-sum decomposition of the tangent space at each point $\theta_x\in TM$:
\begin{equation}
T_{\theta_x}(TM) \simeq VT_{\theta_x}(TM) \oplus HT_{\theta_x}(TM),
\end{equation}
where $VT_{\theta_x}(TM)=\ker(\pi_*)_{\theta_x}$ is called the vertical subspace of $T_{\theta_x}(TM)$, and $HT_{\theta_x}(TM)=\ker K$ is the horizontal subspace of $T_{\theta_x}(TM)$.
It is easy to show that $VT_{\theta_x}(TM)\simeq T_{\theta_x}M$ and $HT_{\theta_x}(TM)\simeq T_{\theta_x}M$. This decomposition is important because of the following fact:
\begin{proposition}
  If a even-dimensional real vector space $V$ can be decomposed into the direct sum $V=V_v\oplus V_h$, where $V_v\simeq V_h$, there is a canonical almost complex structure on $V$, i.e. an automorphism $J:V\to V$ such that $J^2=\text{Id}_V$.
\end{proposition}
This means that there is a canonical almost complex structure on the manifold $T(TM)$.

If, in addition, we have a Riemannian metric $h$ in the original space $M$ (the affine connection $\nabla$ doesn't need to be the Levi-Civita connection of $h$), then there is a natural almost-Hermitian structure on $TM$:
\begin{align}
  g_{\theta_x}((\theta_x, l_x, v_x), (\theta_x', l_x', v_x')) &= h_x(l_x, l'_x) + h_x(v_x, v'_x);\\
  \omega_{\theta_x}((\theta_x, l_x, v_x), (\theta_x', l_x', v_x')) &= h_x(l_x, v_x') - h_x(v_x, l_x');\\
  J_{\theta_x}(\theta_x, l_x, v_x)&=(\theta_x, -v_x, l_x).
\end{align}
This is exactly analogous to the construction of the K\"ahler structure on $\R^{2n}$.

\begin{proposition}
  For a Riemannian manifold $M$ with a dualistic structure, the canonical almost Hermitian structure that we have constructed on $TM$ is K\"ahlerian if and only if the dualistic structure on $M$ is dually flat.
\end{proposition}
In short, a dually flat statistical manifold has a K\"ahlerian tangent bundle.
\begin{corollary}
  The tangent bundle $T\mathcal{E}$ of an exponential family $\mathcal{E}$ is a K\"ahler manifold.
\end{corollary}

% \subsection{Relationship with the Cotangent Bundle}

\bigskip
Finally we show that the symplectic structure on $T\mathcal{E}$ is the same as the canonical symplectic structure on the cotangent bundle $T^*\mathcal{E}$, through the isomorphism given by the Fisher information metric $h$ on $\mathcal{E}$. The exponential connection $\nabla^{(1)}$ on $\mathcal{E}$ is flat, so we can take, at each point $x\in \mathcal{E}$, a local chart $\{q^1,\dots, q^n\}$ with vanishing Christoffel symbols.
Then in the induced chart on the tangent bundle $T\mathcal{E}$, $\{q^1,\dots, q^n,v^1,\dots, v^n\}$ corresponds to the tangent vector $\sum_{i}v^i\frac{\partial}{\partial q^i}\in T_{(q_1,\dots, q_n)}\mathcal{E}$. By Dombrowski's construction, the symplectic form associated to this chart is 
\begin{align}
  \omega_{Dom} = \sum_i h_{ij}\,dv^i\land dq^i,
\end{align}
where $h_{ij}$ are the components of the Fisher information metric.

On the other hand, the canonical symplectic form on the cotangent bundle $T^*\mathcal{E}$ is, locally
\begin{align}
  \omega_{T^*\mathcal{E}} = \sum_{i}dp_i\land dq^i,
\end{align}
and the isomorphism between $T\mathcal{E}$ and $T^*\mathcal{E}$ induced by the metric $h$ sends the canonical symplectic form to
\begin{align}
  \omega_{TM}=\sum_{i}h_{ij}\,dv^i\land dq^i=\omega_{Dom}.
\end{align}
In particular, we see that for any flat connection on $\mathcal{E}$, including the mixture connection, the Dombrowski construction gives the same symplectic form. Therefore for an exponential family, using either exponential connection or mixture connection would give us the same symplectic form.

On the other hand, the almost complex structures and the Riemannian metrics given by different flat conections are different, since the almost complex structure depends on the horizontal subspace determined by the connection. The Riemannian metrics are not even conformally equivalent, in general (e.g take two vectors, one horizontal in exponential connection and one horizontal in mixture connection, with the same horizontal component but different nonzero vertical component).

\section{The Tangent Bundle of a Statistical Manifold}

We consider a statistical manifold $M$ on a measure space $\mathcal{X}$ with measure $\mu$. We assume that probability measures in $M$ are absolutely continuous with respect to $\mu$, and thus we can identify all probability measures in $M$ by their density with respect to $\mu$.
We denote the Fisher information metric on $M$ by $h$, and the $\alpha$-connections by $\nabla^{(\alpha)}$. On the tangent bundle $TM$ of $M$, we denote the almost Hermitian structure obtained by Dombrowski's construction (using $h$ and the exponential connection $\nabla^{(1)}$) by $(g,J,\omega)$.

In \cite{molitor2012remarks}, Molitor notes that for the statistical manifold of all positive probability density functions $\mathcal{P}_n^\times$ on the set of $n$ points (with counting measure), we can define a map $\tau: T\mathcal{P}_n^\times \to \mathbb{P}(\C^n)$ that preserves the almost Hermitian structure of $T\mathcal{P}_n^\times$, where we consider $\Proj(\C^n)$ as a K\"ahler manifold equipped with the Fubini-Study metric. Here we consider a simple generalization of the map for the tangent bundle $TM$, and show that the almost Hermitian structure is preserved under this map. Most of the computation in this section is a simple generalization of Molitor's computation for probability measures on finite point sets in e.g. \cite{molitor2012remarks,molitor2013exponential} with some modifications.

For an element $(p,v)\in TM$ where $v\in T_pM$, let $p_t$ be a curve in $M$ that realizes $v$, i.e. such that $p_0=p$ and $\dot p_t\vert_{t=0}=v$. We can identify $v$ as a function $v:\mathcal{X}\to \R$ given by 
\begin{align}
  \ds p(x) v(x)=\frac{d}{dt}\Big|_{t=0} [p_t(x)]. 
\end{align}
Note that $\ds \int_{\mathcal{X}}v(x)\,p(x)d\mu=0$ if $\ds\int_{\mathcal{X}}\frac{d}{dt}\Big|_{t=0}[p_t(x)]d\mu=\frac{d}{dt}\Big|_{t=0}\int_{\mathcal{X}} p_t(x) d\mu$, which we shall always assume for our $M$.
\bigskip

Let $\mathcal{L}^2(\mathcal{X})$ be the space of complex-valued $L^2$ functions on $\mathcal{X}$ with inner product $\ds \langle f,g\rangle=\int_\mathcal{X} \bar f g \,d\mu$, let $L^2(\mathcal{X})$ be the corresponding Hilbert space (by identifying almost everywhere equal functions), and let $SL^2(\mathcal{X})$ be the unit sphere of $L^2(\mathcal{X})$. 
% In order to state that this map preserves the almost Hermitian structure, we need to first specify what we mean by the Fubini-Study metric on $L^2(\mathcal{X})$.
In general, $L^2(\mathcal{X})$ is an infinite dimensional Hilbert space. We think of $L^2(\mathcal{X})$ as a infinite-dimensional manifold with tangent spaces at each point canonically isomorphic to $L^2(\mathcal{X})$ itself with the same Hermitian inner product.

For each element $q\in \mathcal{L}^2(\mathcal{X})$, we think of it as a mapping $q:\mathcal{X}\to \C$, so we can define functionals $L_x:\mathcal{L}^2(\mathcal{X})\to \C$ for each $x\in \mathcal{X}$ where $L_x(q)=q(x)$.
We define a mapping $f:S\to \mathcal{L}^2(\mathcal{X})$ from an arbitrary manifold $S$ to $\mathcal{L}^2(\mathcal{X})$ to be differentiable if and only the mappings $L_x\circ f:S\to \C$ are differentiable for all $x\in \mathcal{X}$. If the map $f$ is nice (e.g. the condition for the Lebesgue dominated convergence theorem is satisfied for all difference quotients), then composed with the projection $\mathcal{L}^2(\mathcal{X})\to L^2(\mathcal{X})$, the map $S\to L^2(\mathcal{X})$ is differentiable in the Gâteaux sense.

For simplicity, in our formulation we will circumvent issues related to the infinite-dimensionality of the spaces, leaving out analytic details related to general infinite-dimensional manifolds, e.g. general notion of differentiability on Banach manifolds, if irrelevant to our specific examples (see e.g. \cite{ay2017information} for reference on related issues in information geometry), since the more important aspect here is the metrics of the spaces. 

For example if the statistical model given by $M\to \mathcal{P}(\mathcal{X})$ is nice, the mapping $\Phi:TM\to SL^2(\mathcal{X})$ defined by
\begin{align}
  [\Phi(p,v)](x)=\sqrt{p(x)}\exp(\frac{i v(x)}{2}),\,\,\,\,\,\,\,\, x\in \mathcal{X}.%=\sqrt{p(x)}\exp(\frac{i\frac{d}{dt}\Big|_{t=0}p_t(x)}{2p(x)}).
\end{align}
is Gâteaux differentiable. 

Next, we can consider the (non-positive definite) inner products of $\mathcal{L}^2(\mathcal{X})$ as tangent spaces at each point on $\mathcal{L}^2(\mathcal{X})$ as a metric on the infinite-dimensional manifold, and we write $\ds \int_\mathcal{X} \bar q_1 q_2\,d\mu=g_{L^2}(q_1,q_2)+i\omega_{L^2}(q_1,q_2)$. 
% ---/
Since integration is insensitive to measure zero sets, the same metric can be defined on $L^2(\mathcal{X})$ by projection, which is now positive definite and can be thought of as a K\"ahler metric.
% ---
We will then blur the distinction between $\mathcal{L}^2(\mathcal{X})$ and $L^2(\mathcal{X})$; we subsequently understand all point-evaluations and differentiations of curves to take place on $\mathcal{L}^2(\mathcal{X})$, and the evaluations of metric can be projected to $L^2(\mathcal{X})$.

Given the map $\Phi:TM\to L^2(\mathcal{X})$, for two curves $(p_t,v_t),(\tilde p_t, \tilde v_t)$ in $TM$ (where $p_t,\tilde p_t\in M, v_t\in T_{p_t}M, \tilde v_t\in T_{\tilde p_t}M$, where we identify $v,\tilde v$ as the functions $v(x),\tilde v(x)$ on $\mathcal{X}$ as defined before, and $p_0(x) = \tilde p_0(x)=:p(x), v_0(x)=\tilde v_0(x)=:v(x)$), we obtain the pull-back of $g_{L^2}$ and $\omega_{L^2}$ by
\begin{align*}
  &\Phi^* g_{L^2} (\frac{d}{dt}\Big|_{t=0}(p_t, v_t), \frac{d}{dt}\Big|_{t=0}(\tilde p_t, \tilde v_t)) = g_{L^2}(\frac{d}{dt}\Big|_{t=0}\Phi(p_t, v_t), \frac{d}{dt}\Big|_{t=0}\Phi(\tilde p_t, \tilde v_t))\\
  &=\Re \int_{\mathcal{X}}\bigg(\overline{\frac{d}{dt}\Big|_{t=0}\bigg(\sqrt{p_t(x)}\exp(\frac{iv_t(x)}{2})}\bigg)
  \bigg(\frac{d}{dt}\Big|_{t=0}\bigg(\sqrt{\tilde p_t(x)}\exp(\frac{i\tilde v_t(x)}{2})\bigg)\, d\mu
  \stepcounter{equation}\tag{\theequation}
\end{align*}
 Writing $\ds w_1(x)=\frac{1}{2}\Big(\frac{d}{ds}\Big|_{t=0} p_t(x)\Big)\Big/ p(x)$ and $\ds w_2(x)=\frac{1}{2}\Big(\frac{d}{ds}\Big|_{t=0} v_t(x)\Big)$, and similarly $\tilde w_1(x)$ and $\tilde w_2(x)$, we obtain
\begin{align*}
   \Phi^* g_{L^2} (\frac{d}{dt}\Big|_{t=0}(p_t, v_t), \frac{d}{dt}\Big|_{t=0}(\tilde p_t, \tilde v_t)) &= \Re\int_{\mathcal{X}} (\overline{w_1+iw_2})(\tilde w_1+i\tilde w_2)\, p(x)d\mu(x),\\
   &=\int_{\mathcal{X}} (w_1\tilde w_1+w_2\tilde w_2) \,p(x)d\mu(x),
   \stepcounter{equation}\tag{\theequation}\label{L2metric}
\end{align*}
and similarly
\begin{align}
     &\Phi^* \omega_{L^2} (\frac{d}{dt}\Big|_{t=0}(p_t, v_t), \frac{d}{dt}\Big|_{t=0}(\tilde p_t, \tilde v_t)) = \int_{\mathcal{X}} (w_1\tilde w_2 - w_2\tilde w_1) \,p(x)d\mu(x).\label{L2form}
\end{align}

\bigskip
Now we consider the projective space $\Proj L^2(\mathcal{X})\simeq (L^2(\mathcal{X}) - \{0\})/U(1)$ of the Hilbert space $L^2(\mathcal{X})$, and we consider an infinite-dimensional version of the Fubini-Study metric (we imitate the formulation of the metric in \cite{molitor2013exponential}, in which the metric is defined for $\C\Proj^n$).
 % Denote the canonical projection map to be $\pi:L^2(\mathcal{X})-\{0\}\to \Proj L^2(\mathcal{X})$. 
Given a point $z_0\in SL^2(\mathcal{X})$, locally the projective space looks like the orthogonal complement $z_0^\perp$ the subspace $\C z_0$ in $L^2(\mathcal{X})$:
\begin{align}
  \phi_{z_0}:\{[z]\in \Proj L^2(\mathcal{X}) \mid \langle z_0,z\rangle\neq 0\}\to z_0^\perp;\,\,\,\,\,\, [z]\mapsto \frac{1}{\langle z_0, z\rangle}z - z_0.
\end{align}
We define the Fubini-Study metric $g_{FS}$ and symplectic form $\omega_{FS}$ on $\Proj L^2(\mathcal{X})$ by taking the real and imaginary components of the inner products of $z_0^\perp$ respectively, i.e. for $\xi_1,\xi_2\in z_0^\perp$,
\begin{align}
  \langle \xi_1,\xi_2\rangle_{L^2} = ((\phi_{z_0}^{-1})^* g_{FS})(\xi_1,\xi_2) + i  ((\phi_{z_0}^{-1})^* \omega_{FS}) (\xi_1,\xi_2).
\end{align}
% the property that the projection map is a submersion (this is true for finite-dimensional cases: see e.g. 13.4 in \cite{moroianu2007lectures})

Now denote the projection map from $L^2(\mathcal{X})-\{0\}$ to the projective space by $\pi_{\Proj}:L^2(\mathcal{X})-\{0\}\to \Proj L^2(\mathcal{X})$. For two curves $z_1(t),z_2(t)$ on $SL^2(\mathcal{X})$ such that $z_1(0)=z_2(0)=z_0$, we compute the Fubini-Study metric between $\ds \frac{d}{dt}\Big|_{t=0}[z_1(t)]$ and $\ds \frac{d}{dt}\Big|_{t=0}[z_2(t)]$:
\begin{align*}
  &g_{FS}(\frac{d}{dt}\Big|_{t=0}[z_1(t)],\frac{d}{dt}\Big|_{t=0}[z_2(t)]) + i\omega_{FS}(\frac{d}{dt}\Big|_{t=0}[z_1(t)],\frac{d}{dt}\Big|_{t=0}[z_2(t)])\\
  &= \Big\langle\frac{d}{dt}\Big|_{t=0}
  \phi_{z_0}([z_1(t)]),\frac{d}{dt}\Big|_{t=0}
  \phi_{z_0}([z_2(t)])\Big\rangle
  =\Big\langle \dot z_1(0) - z_0\langle z_0, \dot z_1(0)\rangle, \dot z_2(0) - z_0\langle z_0, \dot z_2(0)\rangle\Big\rangle\\
  &= \langle \dot z_1(0),\dot z_2(0)\rangle - \langle \dot z_1(0), z_0\rangle \langle z_0,\dot z_2(0)\rangle.   \stepcounter{equation}\tag{\theequation}\label{fsmetric}
\end{align*}

\bigskip
We now restrict our attention to an exponential family $M$. In particular, the following result is shown by Molitor in \cite{molitor2013exponential}:

\begin{lemma}[Corollary 4.3 and 4.4 in \cite{molitor2013exponential}]
  When $M$ is dually flat (e.g. when $M$ is an exponential family), $TM$ is a K\"ahler manifold with the almost Hermitian structure given by Dombrowski's construction.
\end{lemma}

\bigskip
Now we show that the mapping $\Phi$ composed with the projection $\pi_\Proj: L^2(\mathcal{X})\to \Proj L^2(\mathcal{X})$ preserves the K\"ahler structure:

\begin{proposition}[Compare Proposition 3.3 in \cite{molitor2012remarks}]\label{propfs}
  If $M$ an exponential family, the map $\pi_\Proj\circ\Phi:TM\to \Proj L^2(\mathcal{X})$ is such that $(\pi_{\Proj}\circ\Phi)^* g_{FS}=\frac{1}{4}g, (\pi_{\Proj}\circ\Phi)^*\omega_{FS}=\frac{1}{4}\omega$, where $g$ and $\omega$ are obtained from Dombrowski's construction.
\end{proposition}
\begin{proof}
  First, we compute phase of the map $\Phi$ explicity for an exponential family $M$. Suppose that $(\theta^1, \dots, \theta^n)$ is the natural parameter of the exponential family and the probability density functions in $M$ are of the form
  \begin{align}
    p_\theta(x)=\exp{C(x)+\sum_{k=1}^n\theta^k F_k(x)-\psi(\theta)}
  \end{align}
  for measurable $C(x),F_1(x),\dots, F_n(x)$ on $\mathcal{X}$ such that $1,F_1,\dots, F_n$ are linearly independent, and $\psi$ a function on the parameter space. Now consider $(p,v)\in TM$. Say the tangent vector $v$ based at $p$ has coordinate $(v^k)$ in the natural parameters. Then the function on $\mathcal{X}$ corresponding to $v$ is
  \begin{align*}
    v(x) &= \Big(\frac{d}{dt}\Big|_{t=0} \exp{C(x)+\sum_{k=1}^n(\theta^k+tv^k) F_k(x) -\psi([\theta^k+tv^k])}\Big)\Big/p_\theta(x)\\
    &= \sum_{k=1}^n v^k F_k(x) -\partial_v \psi(\theta).\stepcounter{equation}\tag{\theequation}
  \end{align*}

  % We compute the $g$ and $\omega$ on $TM$ for vertical and horizontal vectors respectively. 
  For every vertical vector on $TM$ based at $(p,v)\in TM$, we can realize it by a curve $(p, v+tu)$ for some $u\in T_pM$. This $u$ is the vertical component of the vector. So it can be easily computed that $w_1=0$ and $w_2(x)=\frac{1}{2} u(x)$.
  Moreover, 
  \begin{align*}
    \Big\langle \frac{d}{dt}\Big|_{t=0}\Phi(p, v+tu), \Phi(p,v)\Big\rangle &= \Big\langle \sqrt{p(x)}\exp(\frac{iv(x)}{2})\frac{iu(x)}{2}, \sqrt{p(x)}\exp(\frac{iv(x)}{2})\Big\rangle\\&=\int_{\mathcal{X}} \frac{iu(x)}{2}\, p(x)d\mu=0.  \stepcounter{equation}\tag{\theequation}
  \end{align*}

  For every horizontal vector $u$ on $TM$ based at $(p,v)\in TM$, we first pick a curve $p_t$ in $M$ such that $\ds \frac{d}{dt}\Big|_{t=0} p_t$ is the horizontal component of the chosen horizontal vector, and $p_0=p$. 
  Now the natural parameters $(\theta_1,\dots, \theta_n)$ of the exponential family $M$ is an affine coordinate of the exponential connection.
  We pick the curve $(p_t, v_t)$ in $TM$ where the coordinate of the tangent vector $v_t$ is constant with respect to $t$ in the natural parameters, so this curve is horizontal. 
  Now 
  \begin{align}
    \frac{d}{dt}\Big|_{t=0} v_t(x)=-\frac{d}{dt}\Big|_{t=0}\partial_v \psi(\theta_t).
  \end{align}
  Note that this is constant on $\mathcal{X}$. So for the horizontal vector $w_1(x)=\frac{1}{2}u(x), w_2(x)\equiv -\frac{1}{2}\frac{d}{dt}\big|_{t=0}\partial_v\psi(\theta_t)$. Moreover,
  \begin{align*}
    &\Big\langle \frac{d}{dt}\Big|_{t=0}\Phi(p_t, v_t), \Phi(p,v)\Big\rangle \\
    &= \Big\langle \sqrt{p(x)}\exp(\frac{iv(x)}{2}) (w_1(x)+iw_2(x))
    , \sqrt{p(x)}\exp(\frac{iv(x)}{2})\Big\rangle\\
    &=\int_{\mathcal{X}} \frac{iu(x)}{2}\, p(x)d\mu- \frac{i}{2}\int_{\mathcal{X}} \frac{d}{dt}\Big|_{t=0}\partial_v \psi(\theta_t)  \, p(x)d\mu=-\frac{i}{2}\frac{d}{dt}\Big|_{t=0} \partial_v \psi(\theta_t).  
    \stepcounter{equation}\tag{\theequation}
  \end{align*}
\bigskip

  We can now compare the Fubini-Study K\"ahler metric with the K\"ahler metric from Dombrowski's construction. We compute $(\pi_{\Proj} \circ \Phi)^* g_{FS}$ and $(\pi_{\Proj} \circ \Phi)^* \omega_{FS}$ on horizontal and vertical vectors using equations (\ref{L2metric}), (\ref{L2form}), (\ref{fsmetric}), and compare them with Dombrowski's construction.
  Let $v, u,\tilde u\in T_pM$, and suppose $\theta_t,\tilde \theta_t$ parameterize curves in $M$ that realize $u,\tilde u$ respectively. Write $C=-\frac{d}{dt}\big|_{t=0}\partial_{v}\psi(\theta_t), \tilde C=-\frac{d}{dt}\big|_{t=0}\partial_{v}\psi(\tilde \theta_t)$. $C$ and $\tilde C$ are constant on $\mathcal{X}$.
  \begin{itemize}
    \item For two vertical vectors $w,\tilde w\in T_{(p,v)}TM$ with vertical components $u,\tilde u\in T_pM$ respectively, 
    \begin{align}
      (\pi_{\Proj} \circ \Phi)^* g_{FS}(w,\tilde w)=\frac{1}{4} \int_{\mathcal{X}} u(x)\tilde u(x) \, p(x)d\mu; \,\,\,(\pi_{\Proj} \circ \Phi)^* \omega_{FS}(w,\tilde w) = 0.
    \end{align}

    \item For two horizontal vectors $w,\tilde w\in T_{(p,v)}TM$ with horizontal components $u,\tilde u\in T_pM$ respectively, 
    \begin{align*}
      (\pi_{\Proj} \circ \Phi)^* g_{FS}(w,\tilde w)&=\frac{1}{4} \bigg(\int_{\mathcal{X}} u(x)\tilde u(x)\, p(x)d\mu +C\tilde C - (iC)(\overline{i\tilde C})\bigg)\\
      &=\frac{1}{4}\int_{\mathcal{X}} u(x)\tilde u(x)\, p(x)d\mu;\\
      (\pi_{\Proj} \circ \Phi)^* \omega_{FS}(w,\tilde w) &= \frac{1}{4}\int_{\mathcal{X}} \big(\tilde C u(x)-C\tilde u(x) \big)\, p(x)d\mu=0.\stepcounter{equation}\tag{\theequation}
    \end{align*}

    \item For a vertical vector $w\in T_{(p,v)}TM$ with vertical component $u\in T_pM$ and a horizontal vector $\tilde w\in T_{(p,v)}TM$ with horizontal component $\tilde u \in T_pM$,
    \begin{align*}
      (\pi_\Proj \circ \Phi)^*g_{FS}(w,\tilde w) &= \frac{1}{4}\int_{\mathcal{X}} \tilde C u(x)\, p(x)d\mu =0;\\
      (\pi_{\Proj} \circ \Phi)^* \omega_{FS}(w,\tilde w) &= \frac{1}{4}\int_{\mathcal{X}} u(x)\tilde u(x)\, p(x)d\mu.
      \stepcounter{equation}\tag{\theequation}
    \end{align*}
  \end{itemize}
  The results all agree with the K\"ahler metric from Dombrowski's construction (other than an $1/4$ factor), so we have proved the proposition.
\end{proof}

This computation establishes a connection between the Fisher metric in information geometry and the Fubini-Study metric in complex geometry.

\bigskip
\section{Statistical Transformation Models and Complex Projective Spaces}

We now consider a statistical transformation model $M$ on $\mathcal{X}$ with acting group $G$, i.e.
\begin{itemize}
  \item $M$ is a statistical manifold on the sample space $\mathcal{X}$;
  \item There is an action of $G$ on the sample space $\mathcal{X}$; 
  \item The action of $G$ on $M$ induced by the pushforward of the action of $G$ on $\mathcal{X}$ is invariant on $M$, i.e. for any probability measure $P\in M$, the probability measure $gP$ defined by
  \begin{align}
    (gP)(A) = P(g^{-1}A)
  \end{align}
  for any measureable set $A$ in $\mathcal{X}$.
  Furthermore we assume the action of $G$ on $M$ is transitive.
\end{itemize}
  % \textit{\textbf{(
  In addition, we assume that $\mathcal{X}$ is a Riemannian manifold with the geometric measure $\mu$ invariant under the action of $G$ on $\mathcal{X}$, and we assume that all probability measures in $M$ are absolutely continuous with respect to the geometric measure.
  % ) (Do we need to assume )}}
  % In addition, we assume that $\mathcal{X}$ has a base measure $\mu$ such that $\mu$ is invariant under the action of $G$ on $\mathcal{X}$, and all probability measures in $M$ are absolutelly continuous with respect to ÷the measure $\mu$. 
  In this way, if $p$ is a density function (with respect to $\mu$) for a probability measure $P$ in $M$, then $p(g^{-1}x)$ is the density function of $gP$ with respect to $\mu$. We usually assume that the density function is smooth on $\mathcal{X}$.

% \bigskip
%   \textbf{\textit{Note: We need to make clear what is our action: we can always make our action to be $gP(A)=P(g^{-1}A)$, or use the post-composition action directly i.e. $gp(x)=p(g^{-1}x)$; these two are in general different, but }}
%   \bigskip

\subsection{The Induced Action on Complex Projective Spaces}

  We consider $L^2(\mathcal{X})$ as a Hilbert manifold as before.
  % There is a map from the unit sphere $SL^2(\mathcal{X})$ to the projective space $\Proj L^2(\mathcal{X})$ given by the projection which we denote by $$
  Now consider the trivial $U(1)$-principal bundle over $\mathcal{X}$. The space of $L^2$ sections of this bundle can be viewed as a group $\mathcal{G}$ with fiber-wise multiplication, and there is an action of $\mathcal{G}$ on the unit sphere $SL^2(\mathcal{X})$ of the Hilbert manifold $L^2(\mathcal{X})$. The orbit space of this action can be identified to be $\mathcal{P}(\mathcal{X})$, the space of all probability density functions over the geometric measure of $\mathcal{X}$, and the orbit projection map $\pi: SL^2(\mathcal{X})\to \mathcal{P}(\mathcal{X})$ can be identified to be $\pi(q)=p\in \mathcal{P}(X)$ where $p(x)=|q(x)|^2$. In particular, consider the subgroup of the bundle $\mathcal{G}$ consisting only of constant sections; this subgroup can be identified with $U(1)$.
  Therefore the projection map $\pi: SL^2(\mathcal{X})\to \mathcal{P}(\mathcal{X})$ can be factorized into a projection $\pi_{\Proj}:SL^2(\mathcal{X})\to \Proj L^2(\mathcal{X})$ and then a map $\pi_{\mathcal{P}}:\Proj L^2(\mathcal{X})\to \mathcal{P}(\mathcal{X})$. 

  Now we consider a lift of the action of $G$ on $M$ to an action of $G$ on the image of the map $\pi_\Proj \circ \Phi$. Note that $\Im\pi_\Proj \circ \Phi\subseteq \pi_\mathcal{P}^{-1} (M)$.
  Write $N=\Im \pi_{\Proj} \circ \Phi$. An element in $N$ can be identified with a class $[q(x)]$ in $\Proj L^2(\mathcal{X})$ where $q(x)\in L^2(\mathcal{X})$.
  For $g\in G$, we let $g[q(x)]=[q(g^{-1}x)]$. This is well-defined, and respects the group action. We will further show that $g[q(x)]\in N$ for every $g\in G, [q(x)]\in N$. For this, we consider the action of $G$ on $TM$ given by $g(p,v)=(gp, g_*v)$ where $(p,v)\in TM$ and we regard $g$ as a diffeomorphism on $M$ and write $g_*v$ to be pushforward of $v$ under $g$.

%%%%%%%%%%%%%%%%%%%%%%%%%%%
  % \bigskip
  % \textbf{\textit{Need to characterize $N$ geometrically!!}}
  % \bigskip
%%%%%%%%%%%%%%%%%%%%%%%%%%%

  \begin{lemma} $N$ is invariant under the action of $G$, and the map $\pi_{\Proj}\circ \Phi$ is $G$-equivariant.
  \end{lemma}
  \begin{proof}
    Given $(p,v)\in TM$ and $g\in G$, let $p_t$ be a curve in $M$ realizing $v$. The tangent vector $g_*v\in T_{gp}M$ can be identified with the function 
    \begin{align}
      (g_* v)(x) = \frac{\frac{d}{dt}\big|_{t=0} gp_t(x)}{gp(x)}=\frac{\frac{d}{dt}\big|_{t=0}p_t(g^{-1}x)}{p(g^{-1}x)}=v(g^{-1}x).
    \end{align}
    Now $\ds \Phi(g(p,v))=\sqrt{p(g^{-1}x)}\exp(\frac{iv(g^{-1}x)}{2})=g\Phi(p,v)$. This shows that $N=\Im \pi_{\Proj}\circ\Phi$ is invariant under the action of $G$ and that $\pi_{\Proj}\circ\Phi$ is a $G$-equivariant map.
  \end{proof}

  We notice that $G$ acts by isometry (with respect to Fisher metric) on $M$:

  \begin{lemma}[Chapter 3.8 of \cite{barndorff2012parametric}, or Proposition 2.2 and 2.3 of \cite{fernandes2000fisher}]
    The action of $G$ on $M$ induced from $G$ acting on $\mathcal{X}$ preserves both the Fisher metric and the exponential connection. \label{lemmaActByIsom}
  \end{lemma}

  Motivated by this lemma and Proposition \ref{propfs}, we consider the Fubini-Study metric of $N$ under the action of $G$. For this we first consider the action of $G$ on $TM$.

  \begin{lemma}
    The action of $G$ on $TM$ preserves the K\"ahler metric given by Dombrowski's construction.\label{lemmaTMpreserved}
  \end{lemma}
  \begin{proof}
    For a vertical vector $u^v\in T_{(p,v)}TM$ with vertical component $u\in T_pM$, we consider its horizontal and vertical component after the pushforward of the diffeomorphism $g\in G$. Denote the projection map of the tangent bundle to be $\pi: TM\to M$. The vertical vector $u^v$ can be realized by a curve $(p,v+tu)$ in $TM$. The horizontal component of the pushforwarded vector is
    \begin{align}
      \pi_* g_* (u^v) = \frac{d}{dt}\Big|_{t=0} \pi(g(p,v+tu))=0.
    \end{align}
    Let $K:TTM\to TM$ be the connector map for the exponential connection of $M$. The vertical component of the pushforwarded vector is
    \begin{align}
      K\big(g_*(u^v)\big) = K\big(\frac{d}{dt}\Big|_{t=0}(gp,g_*v+tg_*u)\big)=g_*u.
    \end{align}
    Therefore $g_*(u^v)$ is a vertical vector with vertical component $g_*u$.

    For a horizontal vector $u^h\in T_{(p,v)}TM$ with horizontal component $u\in T_pM$, $u^h$ can be realized by the horizontal lift to $TM$ of a curve $p_t$ in $M$ that realizes $u$. Denote the lifted curve to be $(p_t,v_t)\in TM$.
    % we consider its horizontal and vertical component after the pushforward of the diffeomorphism $g\in G$. Denote the projection map of the tangent bundle to be $\pi: TM\to M$. T
    The horizontal component of the pushforwarded vector is
    \begin{align}
      \pi_* g_* (u^h) = \frac{d}{dt}\Big|_{t=0} \pi(g(p_t, v_t))=\frac{d}{dt}\Big|_{t=0}gp_t=g_*u.
    \end{align}
    The vertical component of the pushforwarded vector is 
    \begin{align}
      K\big(g_*(u^h)\big) =K\frac{d}{dt}\Big|_{t=0} g(p_t,v_t)=\frac{D^{(e)}}{dt}\Big|_{t=0} g(p_t,v_t)=g_*\frac{D^{(e)}}{dt}\Big|_{t=0}(p_t,v_t)=0
    \end{align}
    where $\frac{D^{(e)}}{dt}$ is the covariant derivative for the exponential connection on $M$, and the second to last equality holds because as a diffeomorphism of $M$, $g$ is an affine transformation with respect to the exponential connection (Lemma \ref{lemmaActByIsom}). Therefore $g_*(u^h)$ is a horizontal vector with horizontal component $g_*u$.

    Finally, since by Lemma \ref{lemmaActByIsom} $G$ acts on $M$ by isometry, we know that $h_{gp}(g_*u, g_*\tilde u)=h_p(u,\tilde u)$ for $h$ the Fisher metric on $M$, $g\in G$, and $u,\tilde u\in T_pM$.
    By Dombrowski's construction of the K\"ahler metric on $TM$, the action of $G$ on $TM$ preserves the K\"ahler metric.
  \end{proof}

  Now this action of $G$ on $TM$ preserving K\"ahler metric transfers, via the $G$-equivariant map $\pi_\Proj\circ\Phi$ which also preserves the K\"ahler structure (up to a constant factor), to an action of $G$ on $N$. Write the Fubini-Study K\"ahler metric as $G_{FS}=g_{FS}+i\omega_{FS}$, and the K\"ahler metric from Dombrowski's construction as $G_{Dom}=g_{Dom}+i\omega_{Dom}$.

  \begin{corollary}
    The action of $G$ on $N$ preserves the Fubini-Study K\"ahler metric, in the sense that $G_{FS}(g_*w,g_*\tilde w)=G_{FS}(w,\tilde w)$ for all $w,\tilde w\in  (\pi_{\Proj}\circ \Phi)_*\big[T_{(p,v)}TM\big], g\in G, (p,v)\in TM$.
  \end{corollary}
  \begin{proof}
    Given two $u,\tilde u\in T_{(p,v)}TM$, from Proposition \ref{propfs} we know that 
    \begin{align}
      G_{FS}((\pi_{\Proj}\circ\Phi)_* u),((\pi_{\Proj}\circ\Phi)_*\tilde u))=(\pi_\Proj\circ\Phi)^*G_{FS}(u,\tilde u)=\frac{1}{4}G_{Dom}(u,\tilde u)
    \end{align}
    and that 
    \begin{align}
    G_{FS}((\pi_{\Proj}\circ\Phi)_*(g_*u), (\pi_{\Proj}\circ\Phi)_*(g_*\tilde u))=
    (\pi_{\Proj}\circ\Phi)^*G_{FS}(g_*u,g_*\tilde u) =\frac{1}{4}G_{Dom}(g_*u,g_*\tilde u).       
    \end{align}
    But by Lemma \ref{lemmaTMpreserved} $G_{Dom}(g_*u,g_*\tilde u)=G_{Dom}(u,\tilde u)$. Finally, since $\pi_\Proj\circ\Phi$ is $G$-equivariant,
    \begin{align}
      (\pi_\Proj\circ\Phi)_*(g_*u)= (\pi_\Proj\circ\Phi\circ g)_*u=(g\circ\pi_\Proj\circ\Phi)_*u=g_*(\pi_\Proj\circ\Phi)_*u.
    \end{align}
    and similarly for $\tilde u$. Therefore we can conclude that
    \begin{align}
      G_{FS}(g_*w,g_*\tilde w)=G_{FS}(w,\tilde w),
    \end{align}
    for all $w,\tilde w\in  (\pi_{\Proj}\circ \Phi)_*\big[T_{(p,v)}TM\big]$. 
  \end{proof}

% We now consider the structure of $N$.

% \subsection{Computations related to the Symplectic Structure on \texorpdfstring{$TM$}{TEXT}}

% Let $(M,h)$ be an exponential family and $(TM, g,J,\omega)$ its tangent bundle, where $h$ is the Fisher information metric and $(g,J,\omega)$ is the K\"ahler structure via Dombrowski construction.
% % For simplicity, we assume contractibility of the statistical manifold $M$, and thus $TM$ is contractible. There is then an $1$-form $\theta$ such that $\omega=d\theta$.
% Supposing that there is a comoment map $\lambda^*:\mathfrak{g}\to C^\infty(TM)$; then denoting the fundamental vector field of $X$ by $X^*$, property of $\lambda^*$ implies that $X^*$ is the Hamiltonian vector field of $\lambda^*(X)$:
% \begin{align}
%   \iota_{X^*}\omega=d\lambda^*(X).
% \end{align}
% In a chart $(q^1,\dots, q^n)$ of $M$, which induces a chart $(q^1,\dots, q^n,p_1,\dots, p_n)=\sum p_i(\frac{\partial}{\partial q^i})_{(q^1,\dots, q^n)}$ of $TM$, say 
% \begin{align}
%   X^* = 
% \end{align}

\section{Moment Maps and The Orbit Method}

We first briefly recall the constructions of induced representations and Kirillov's orbit method that gives a correspondence between the coadjoint orbits of a group and unitary irreducible representations of the group.  We then proceed to prove that in certain situations, the left-translation unitary representation of $G$ in $L^2(\mathcal{X})$ contains as subrepresentations unitary irreducible representations corresponding to the coadjoint orbits in the image of the moment map of the action of the group on the tangent bundle of an exponential transformation model (if the action is Hamiltonian).

\subsection{Induced Representation}

We describe the construction of induced representations in a special situation pertaining to our setting (see \cite{corwin2004representations}; for the general version of this construction, see \cite{folland2016course}). 
Given a locally compact group $G$ and a closed subgroup $H\leq G$, we assume the homogeneous space $G/H$ possesses an invariant measure $\mu$ (which is indeed the assumption for a statistical transformation model). Suppose $H$ has a unitary representation $\sigma$ on the Hilbert space $\mathcal{H}_\sigma$. We want to construct out of this representation $\sigma$ on $H$ a unitary representation $\pi=\text{ind}_H^G(\sigma)$ on $G$. We first construct a Hilbert space $\mathcal{H}_\pi$ consisting of all Borel measurable functions $f:G\to \mathcal{H}_\sigma$ such that:
\begin{itemize}
  \item Contravariance: for any $h\in H, g\in G, f(gh)=\sigma(h^{-1})f(g)$;
  \item Square integrable: 
  \begin{align}\label{square_integrable}
    \int_{G/H} \norm{f(g)}^2 \, d\mu(gH) <\infty.
  \end{align}
  Note that $\norm{f(g)}^2$ is constant on cosets by the contravariance property, so the integral (\ref{square_integrable}) is well-defined.
\end{itemize}
% \begin{align}
%    \mathcal{H}_\pi = \{f:G\to \mathcal{H}_\sigma: f\text{ Borel measurable and for any }h\in H, g\in G, f(hg)=\sigma(h)f(g)\}.
%  \end{align} 
 We define an inner product on the set $\mathcal{H}_\pi$:
 \begin{align}
   \langle f, \tilde f\rangle = \int_{G/H} \langle f(g), \tilde f(g)\rangle_{\mathcal{H}_\sigma} \, d\mu (gH).
 \end{align}
 Again, this integral is well-defined by the contravariance property. Thus the set $\mathcal{H}_\pi$ is made into a Hilbert space, and then we can define the representation $\pi=\text{ind}_H^G(\sigma)$ by
 \begin{align}\label{lefttransrep}
   [\pi(g) f](x)= f(g^{-1} x).
 \end{align}
 for any $g\in G, x\in G/H$. By the $G$-invariance property of the measure $\mu$, $\pi$ is a unitary representation of $G$ with Hilbert space $\mathcal{H}_\pi$.

 \subsection{Coadjoint Orbits and Kirillov's Orbit Method}\label{orbitMethod}

Let $G$ be a nilpotent Lie group.
Let $\mathfrak{g}^*$ be the dual Lie algebra of $G$. Kirillov's orbit method gives a bijective correspondence between orbits of the coadjoint action of $G$ on $\mathfrak{g}^*$ and unitary irreducible representations of $G$. We describe this correspondence, and for details of Kirillov's orbit method we refer to e.g. \cite{kirillov2004lectures,corwin2004representations,folland2016course} (the presentation of this section mainly follows \cite{folland2016course}).

Let $\lambda \in \mathfrak{g}^*$. We write $\mathcal{O}_\lambda$ for its coadjoint orbit. Let $\mathfrak{h}$ be a maximal subalgebra in $\mathfrak{g}$ such that $\lambda$ restricts to $0$ on $[\mathfrak{h},\mathfrak{h}]$ (see e.g. Theorem 2.2.1 in \cite{corwin2004representations} for the existence of such a subalgebra). We can then define a one-dimensional representation $\sigma_\lambda$ as follows: for a nilpotent Lie group $G$, $H=\exp \mathfrak{h}$ is a connected subgroup with Lie algebra $\mathfrak{h}$ ($\exp$ is an analytic diffeomorphism for a nilpotent Lie group); let $\sigma_\lambda:H\to \C^*$ be defined as
\begin{align}\label{1dunirep}
  \sigma_\lambda(\exp X) = e^{2\pi i \lambda(X)},
\end{align}
for any $\exp X\in H, X\in \mathfrak{h}$.
This is indeed a unitary representation by the fact that $\lambda\vert_{[\mathfrak{h},\mathfrak{h}]}=0$ and Baker–Campbell–Hausdorff formula.

Then we can associate with $\lambda$ the representation $\text{ind}_{H}^G(\sigma_\lambda)$ on $G$. Our construction depends on a choice of maximal subalgebra $\mathfrak{h}\subseteq \mathfrak{g}$ such that $\lambda$ vanishes on $[\mathfrak{h},\mathfrak{h}]$, but in fact the choice of such $\mathfrak{h}$ does not alter the equivalence class of the representation $\text{ind}_H^G(\sigma_\lambda)$ (see e.g. Theorem 2.2.2 in \cite{corwin2004representations}). Moreover, the maximality property of the subalgebra $\mathfrak{h}$ gives the irreducibility of the representation $\text{ind}_H^G(\sigma_\lambda)$, and for any two elements in $\mathfrak{g}^*$ in the same coadjoint orbit the resulting representations are equivalent. Finally, every equivalence class of unitary irreducible representation of $G$ can be realized by this construction. Summarizing, we have the following:
\begin{theorem}[see e.g. Theorem 7.9 in \cite{folland2016course}]
  Let $G$ be a simply connected nilpotent Lie group. There is a correspondence between coadjoint orbits of $G$ and unitary irreducible representations of $G$ given by $\mathcal{O}_\lambda\mapsto \text{ind}_H^G(\sigma_\lambda)$.
\end{theorem}

\subsection{The Moment Map on Tangent Bundles of Transformation Models}

\begin{theorem}
  Let $G$ be a simply connected nilpotent Lie group. Let $\mathcal{X}$ be a $G$-homogeneous space with a $G$-invariant measure, such that the stabilizer subgroup $N$ at some point in $\mathcal{X}$ is normal, i.e. $\mathcal{X}=G/N$.  For any $\lambda \in \mathfrak{g}^*$ such that $\lambda|_{\text{Lie}(N)}=0$, the irreducible unitary representation of $G$ corresponding to the coadjoint orbit $\mathcal{O}_\lambda$ under Kirillov's correspondence can be realized as a sub-representation of the left-translation representation of $G$ on $L^2(\mathcal{X})$. 

\end{theorem}

  \begin{proof}
    Let $\lambda\in \mathfrak{g}^*$. Denoter $\text{Lie}(N)=\mathfrak{n}$. Then $\lambda\vert_{\mathfrak{n}}=0$. Let $\mathcal{S}$ be the collection of subalgebras $\mathfrak{h}$ in $\mathfrak{g}$ such that $\lambda\vert_{[\mathfrak{h},\mathfrak{h}]}=0$ and $\mathfrak{n}\subseteq \mathfrak{h}$. $\mathcal{S}$ is nonempty since $\mathfrak{n}\in \mathcal{S}$. Moreover if $\mathfrak{h}_1\subseteq \mathfrak{h}_2\subseteq\dots$ is a chain in $\mathcal{S}$, then $\bigcup_{n} \mathfrak{h}_n$ is in $\mathcal{S}$ and is an upper bound of the chain. By Zorn's lemma $\mathcal{S}$ contains a maximal element $\mathfrak{h}$ with respect to inclusion, which is also a maximal subalgebra in $\mathfrak{g}$ such that $\lambda\vert_{[\mathfrak{h},\mathfrak{h}]}=0$. Therefore we can use $\mathfrak{h}$ and the corresponding Lie subgroup $H=\exp(\mathfrak{h})$ to build an induced representation.

    Following the notations in Section \ref{orbitMethod}, we denote by $\sigma_\lambda$ the $1$-dimensional unitary representation of the subgroup $H$ given by the formula (\ref{1dunirep}). The construction in Section \ref{orbitMethod} requires an invariant measure on $G/H$.
    In our case, $N$ is a subgroup of $H$, and we have assumed that $G/N$ possesses an invariant measure which we denote by $\mu$. Let $\rho:G/N\to G/H$ be the natural projection map, and $\rho_*(\mu)$ the pushforward measure on $G/H$ with the property that for any measurable function $f$ on $G/H$
    \begin{align}\label{pushforwardmeas}
      \int_{G/H} f(gH) \, d(\rho_* \mu)(gH)= \int_{G/N} f\circ \rho(gN) \, d\mu (gN).
    \end{align}
    This measure is $G$-invariant, since for any $g\in G$, 
    \begin{align}
      g_*(\rho_* \mu)=(g\circ \rho)_*\mu = (\rho\circ g)_*\mu=\rho_*(g_*\mu)=\rho_*\mu.
    \end{align}
    Now the Hilbert space $\mathcal{H}$ of the induced representation $\text{ind}_H^G(\sigma_\lambda)$ in our case contains all Borel measurable functions $f:G\to \C$ such that for any $h\in H, g\in G$, $f(gh)=\sigma(h^{-1})f(g)$ and $\int_{G/H} \norm{f(g)}^2 d(\rho_*\mu)(gH)<\infty$. 
    % Now by (\ref{pushforwardmeas}), the integrability property is equivalent to $\int_{G/N} $
    Now for any $\exp X\in N$ for $X\in \mathfrak{n}$, since $\lambda(X)=0$, we know that $\sigma(\exp X)=1$. Therefore for any $n\in N, g\in G$, $f(gn)=\sigma(n^{-1})f(g)=f(g)$. Therefore for any $f\in \mathcal{H}$, $f$ is constant on each coset of $N$. Thus $f$ induces a Borel measurable function $\tilde f:G/N\to \C$. Moreover, by (\ref{pushforwardmeas}), the integrability property for functions in $\mathcal{H}$ is equivalent to $\int_{G/N} \norm{f(g)}^2 \, d\mu (gN)<\infty$.  Therefore $\tilde f \in L^2(G/N)$. Therefore there is an embedding of $\mathcal{H}$ to the Hilbert space $L^2(G/N)$, and the unitary representation of $G$ on $\mathcal{H}$ defined by (\ref{lefttransrep}) coincides with the natural left-translation representation of $G$ in $L^2(G/N)$. We can then conclude that the representation $\text{ind}_H^G(\sigma_\lambda)$ can be realized as a sub-representation of the left-translation representation of $G$ on $L^2(\mathcal{X})=L^2(G/N)$.
  \end{proof}
\bigskip

\begin{corollary}
  Under the same assumptions of $G$ and $\mathcal{X}$, let $M$ be an exponential transformation model on $\mathcal{X}$, and $TM$ the K\"ahler manifold that is the tangent bundle of $M$. 
  If there is a map $\mu:TM\to \mathfrak{g}^*$ such that for any $\lambda \in \text{Im}[\mu]\subseteq \mathfrak{g}^*$, $\lambda|_{\text{Lie}(N)}=0$, then the irreducible unitary representation of $G$ corresponding to the coadjoint orbit $\mathcal{O}_\lambda$ under Kirillov's correspondence using the orbit method can be realized as a sub-representation of the left-translation representation of $G$ on $L^2(\mathcal{X})$. Moreover these subrepresentations corresponding to different coadjoint orbits in the image of $\lambda$ form a direct sum in $L^2(\mathcal{X})$.  
\end{corollary}
\begin{proof}
  The direct sum property is given by the fact that irreducible subrepresentations in $L^2(\mathcal{X})$ have no non-trivial proper subrepresentations.
\end{proof}
\bigskip

      In particular, suppose the symplectic $G$-action on $TM$ is Hamiltonian with moment map $\mu:TM\to\mathfrak{g}^*$.
Let $\mu^*:\mathfrak{g}\to C^\infty(TM)$ be the comoment map of $\mu$.
If $X\in \mathfrak{n}=\text{Lie}(N)$, then its fundamental vector field $X^\sharp$ on $TM$ is $0$, so 
        \begin{align*}
            d\mu^*(X)=\iota_{X^\sharp}\omega_{Dom}=0\implies \mu^*(X)\text{ is constant}.
        \end{align*}
    % \begin{itemize}
        % \item 
        % \item 
         If there exists $c\in [\mathfrak{g},\mathfrak{g}]^0\subseteq \mathfrak{g}^*$ such that $c|_{\mathfrak{n}}=\mu^*|_{\mathfrak{n}}$, then the dual map $(\mu^*-c)^*:TM\to \mathfrak{g}^*$ of $(\mu^*-c)$ is a moment map. 
Moreover $(\mu^*-c)^*$ satisfies the condition in the theorem: for any $\lambda\in \text{Im}[(\mu^*-c)^*]$, $\lambda|_{\mathfrak{n}}=0$.
         % \end{itemize} 
In other words, there exists a moment map $\tilde \mu:TM\to \mathfrak{g}^*$ of the Hamiltonian $G$-action on $TM$ such that the conditions of the corollary holds.
    To summarize,
    \begin{corollary}
      Under the same assumptions of $G$,$\mathcal{X}$, $M$, if the $G$-action on $TM$ is Hamiltonian with comoment map $\mu^*:\mathfrak{g}\to C^\infty(TM)$ and $c\in [\mathfrak{g},\mathfrak{g}]^0$ such that $c|_{\text{Lie}(N)}=\mu^*|_{\text{Lie}(N)}$, then there exists a moment map $\tilde \mu:TM\to \mathfrak{g}^*$ of the $G$-action such that
               \begin{align*}
             \bigoplus_{\mathcal{O}_\lambda,\lambda\in \text{Im}[\tilde \mu]}\pi_{\mathcal{O}_\lambda}\subseteq L^2(\mathcal{X})
         \end{align*}
         is a subrepresentation of the unitary $G$-representation in $L^2(\mathcal{X})$, where $\pi_{\mathcal{O}_\lambda}$ is a $G$-subrepresentation in $L^2(\mathcal{X})$ equivalent to the unitary irreducible representation of $G$ corresponding to $\mathcal{O}_\lambda$ under Kirillov's correspondence.
    % \end{itemize}
    \end{corollary}

Since for any comoment map $\mu^*$, $\mu^*(X)$ is constant, such comoment map is already close to satisfying the condition that any $\lambda\in \text{Im}[\mu]$ satisfies $\lambda|_{\text{Lie}(N)}=0$. Therefore other alternative conditions for the Hamiltonian $G$-action can be required to arrive at the same conclusion.
\bigskip

\nocite{*}

\bibliographystyle{unsrt}

% \bibliography{TransformModels.bib}

\end{document}